\documentclass[10pt, twocolumns]{IEEEtran}

\usepackage{cite}
\usepackage{amsmath,amssymb,amsfonts,amsthm}
\usepackage{algorithmic}
\usepackage{textcomp}
\usepackage{color}
\usepackage{bm}
\usepackage{graphicx}
\usepackage{enumerate}
\usepackage{multicol}
\usepackage{algorithm}
\usepackage{booktabs,tabularx,makecell,array}
\usepackage{tikz}
\usepackage{circuitikz}
\usetikzlibrary{arrows.meta, positioning}
\usepackage{pgfplots}
\pgfplotsset{compat=1.18}
\usepackage{subcaption}

\newcolumntype{L}{>{\raggedright\arraybackslash}X}

\newtheorem{lemma}{Lemma}

\newtheorem{proposition}{Proposition}
\newtheorem{remark}{Remark}

\newtheorem{theorem}{Theorem}

\newtheorem{problem}{Problem}
\newtheorem{definition}{Definition}
\def\begmat#1{\begin{bmatrix}#1\end{bmatrix}}

\def\rea{\mathbb{R}}

\title{
\huge{Data-Driven Linear Quadratic Control Using Output-Feedback via Non-Minimal Realization}
}

\author{Weijian~Li, %\emph{Member, IEEE},
Bowen Yi, %\emph{Member, IEEE},
Panos J. Antsaklis, %\emph{Fellow, IEEE},
and
Hai Lin %\emph{Senior Member, IEEE}

\thanks{This work was supported in part by the Natural Sciences and Engineering Research Council of Canada (NSERC) under Grant RGPIN-2024-0478, the Fonds de recherche du Québec (FRQ) under Grant 378208, and the Programme PIED of Polytechnique Montr\'eal.}

\thanks{Weijian~Li, Panos J. Antsaklis and Hai Lin are with the Department of Electrical Engineering, College of Engineering, University of Notre Dame, Notre
Dame, IN 46556 USA 	
(e-mails: wli26@nd.edu;
hlin1@nd.edu;
pantsakl@nd.edu).}

\thanks{Bowen Yi is with the Department of Electrical Engineering, Polytechnique Montréal and GERAD, Montréal, QC H3T 1J4, Canada (e-mail: bowen.yi@polymtl.ca).}
}

\begin{document}

\maketitle

%%%%%%
\begin{abstract}
In this paper, we investigate a continuous-time linear quadratic control problem for systems with unknown matrices, where only input-output data are available. We propose an output-feedback learning framework based on a canonical non-minimal realization constructed through Kreisselmeier’s adaptive filter. The filter admits an observer interpretation, which leads to an augmented system that preserves the input-output response of the realization and provides accessible state trajectories.
We show that the optimal gain of this augmented system explicitly recovers the optimal gain associated with the canonical non-minimal realization, and hence achieves the optimal state-feedback solution of the original plant.
Exploiting this relation and the known structure of the augmented input matrix, we develop a data-driven value iteration algorithm within the adaptive dynamic programming framework. The resulting controller is implementable from input-output data, and its performance is validated via simulations.
%The filter provides an accessible auxiliary trajectory and admits an observer interpretation for the non-minimal realization.
\end{abstract}


%%%%
%%%%-----------------
\section{Introduction}

In recent years, reinforcement learning (RL) and approximate/adaptive dynamic programming (ADP) have attracted significant attention for solving optimal control problems of unknown dynamical systems  \cite{sutton2018reinforcement, werbos1974beyond}. This is motivated by their applications in a wide range of areas, including autonomous driving \cite{kiran2021deep}, robotics \cite{kober2013reinforcement}, and power systems \cite{markovsky2023data}. 
In particular, there has been renewed interest in a
fundamental instance, namely the linear quadratic regulator (LQR) problem, which involves controlling an unknown linear system to minimize an infinite-horizon quadratic cost. This problem is not only theoretically tractable, but also serves as a benchmark for feedback controller design in engineering systems \cite{recht2019tour, lewis2009reinforcement, jiang2017robust}.

It is well-known that the optimal control law for an LQR problem is characterized by a static linear state-feedback policy \cite{lewis2012optimal}.
When the system model is known, the control gain can be obtained by solving an algebraic Riccati equation (ARE).
However, this approach relies on exact knowledge of the system model, and becomes computationally expensive for large-scale systems.
To address these issues, the policy iteration (PI) method was proposed to iteratively solve the ARE and compute the control gain \cite{kleinman1968iterative}, and it was later generalized to partially and fully unknown systems in \cite{vrabie2009adaptive, jiang2012computational}.
While PI typically enjoys fast convergence, it requires an initial stabilizing gain and solves a Lyapunov equation at each iteration. 
To overcome these limitations, a value iteration (VI) method was subsequently developed in \cite{bian2016value} with convergence guarantees.
In parallel, policy optimization methods inspired by RL, such as  gradient
descent, natural policy gradient, and Newton-type algorithms, have also been studied for the LQR problem \cite{fazel2018global, mohammadi2021convergence}.
Nevertheless, the aforementioned works require access to the system state.

In many applications, only system outputs, rather than states, are available \cite{rizvi2023output}. Since static output-feedback controllers generally fail to solve the LQR problem \cite{fatkhullin2021optimizing}, 
recent efforts have been devoted to the linear quadratic control problem using dynamic output-feedback approaches.
For instance, in a discrete-time setting, the system state was reconstructed from a finite segment of input-output trajectory,  based on which PI and VI algorithms were designed in \cite{lewis2010reinforcement}. 
Then, an output-feedback Q-learning scheme was proposed in \cite{rizvi2018output} without the knowledge of the system model. In \cite{duan2023optimization}, dynamic output feedback controllers were investigated on the optimization landscape.
In the continuous-time context, adaptive suboptimal output-feedback control for linear systems with known input matrix was investigated via integral reinforcement learning in \cite{zhu2014adaptive}.
In \cite{rizvi2019reinforcement, rizvi2023output}, a model-free observer was constructed, and then PI and VI methods were subsequently proposed.
The idea was further employed for optimal output tracking in \cite{chen2022adaptive}. Following that, data-driven output-feedback ADP was explored with relaxed assumptions in \cite{lin2025new}. However, accurate state reconstruction requires the system to run for a sufficiently long period before data collection.

Very recently, the canonical non-minimal realization framework developed in \cite{bosso2025data, bosso2026data} has emerged as a powerful tool for data-driven control of continuous-time linear time-invariant (LTI) systems using only input-output trajectory data. It provides an auxiliary higher-dimensional state-space representation that can be constructed from input-output data without identifying a minimal realization of the plant. This idea was further developed in \cite{gao2025input}, where Kreisselmeier's adaptive filter was employed to generate an accessible non-minimal state trajectory directly from measured input-output signals, enabling data-driven stabilization. 

Motivated by these results, this paper develops a data-driven output-feedback linear quadratic control method for continuous-time LTI systems with unknown matrices. The proposed method avoids explicit state reconstruction by learning the controller in an auxiliary coordinate generated by Kreisselmeier's adaptive filter, thereby removing the need for the system to run for a sufficiently long period before data collection. The main contribution is a framework based on a canonical non-minimal realization, which reformulates the original problem as an augmented LQR problem with accessible state trajectories. An explicit relation between the optimal gain of the augmented problem and that of the original plant further enables a data-driven VI algorithm with a relaxed rank condition and lower per-iteration computational complexity.

The rest of this paper is organized as follows. 
In Section \uppercase\expandafter{\romannumeral2}, we formulate a linear quadratic control problem and present necessary preliminaries on VI algorithms.
In Section \uppercase\expandafter{\romannumeral3}, we present the proposed framework based on a canonical non-minimal realization. Then, we design a data-driven VI algorithm in Section \uppercase\expandafter{\romannumeral4}, and provide simulation results in Section \uppercase\expandafter{\romannumeral5}. Finally, we conclude this paper in Section \uppercase\expandafter{\romannumeral6}.

\textbf{Notation:}
Let 
%$\mathbb R^m$ be the set of $m$-dimensional real column vectors, $\mathbb R^{m \times n}$ be the set of $m$-by-$n$ dimensional real matrices, 
$\mathbb N$ be the set of nonnegative integers, and $\mathbb C_+$ be the set of complex numbers with non-negative real parts.
$\mathbb S^{n}$, $\mathbb S_+^{n}$ and  $\mathbb S_{++}^{n}$ denote the sets of $n$-by-$n$ real symmetric matrices, positive semi-definite matrices, and positive definite matrices, respectively.
We use $0_m$ ($1_m$) to represent the $m$-dimensional vector with all entries of $0$ ($1$).
%, and $I_n$ denote the $n$-by-$n$ identity matrix.
%We simply use $0$ to denote a zero vector or matrix of appropriate dimension when there is no confusion.
% Denote by $(\cdot)^\top$, $\otimes$ and ${\rm rank}(\cdot)$ the transpose, the Kronecker product and the rank of a matrix. 
Let $\Vert \cdot \Vert$ be the Euclidean norm for vectors and the induced $l_2$-norm for matrices, and $\Vert \cdot \Vert_F$ be the Frobenius norm.
For $x \in \mathbb{R}^n$, $\mathrm{diag}(x) \in \mathbb{R}^{n \times n}$ is the diagonal matrix with the entries of $x$.
For $A \in \mathbb R^{m \times n}$, ${\rm vec}(A) = [a_1^\top, a_2^\top, \dots, a_n^\top]^\top$, where $a_i$ is the $i$-th column of $A$.

%%%%
%%%%----------------- 
\section{Formulation and Preliminary}

In this section, we introduce a linear quadratic control problem, and present a model-based VI algorithm.

\subsection{Problem Statement}

Consider a continuous-time LTI system described by
\begin{equation}
\label{dyn:LTI}
\dot x = Ax + Bu, ~ y = Cx,
\end{equation}
where $x \in \mathbb R^n$ is the system state with initial condition $x(0) = x_0$, $u \in \mathbb R^m$ is the control input, $y \in \mathbb R^p$ is the measured output, and $A$, $B$ and $C$ are constant matrices with compatible dimensions.
Suppose that the pair $(A, B)$ is controllable, and the pair $(C, A)$ is observable.

The linear quadratic control aims to find a control policy that minimizes an infinite-horizon quadratic cost, i.e.,
\begin{equation}
\begin{aligned}
\label{form}
\min~\int_0^\infty (y^\top Q y + u^\top R u) \mathrm{d} t,
~~{\rm subject~to}~\eqref{dyn:LTI},
\end{aligned}
\end{equation}
where $Q \in \mathbb S_{+}^p$ and $R \in \mathbb S_{++}^m$ are user-defined weighting matrices such that the pair $(\sqrt{Q}C, A)$ is observable.

If the model \eqref{dyn:LTI} is known and the state $x$ is measurable, then \eqref{form} reduces to a standard LQR problem. It is well-known that the optimal state-feedback controller
is
$u(t) = -K^*x(t)$,  where $K^* = R^{-1} B^\top P^*$, and $P^*$ is the unique positive definite solution to the ARE \cite{lewis2012optimal}
\begin{equation}
\label{ARE}
A^\top P + PA - PBR^{-1} B^\top P + C^\top Q C = 0.
\end{equation}

However, in many practical scenarios, the state $x$ is not available, and the system matrices $(A,B,C)$ are unknown \cite{rizvi2023output}. This motivates the following problem.

\begin{problem}
\label{prob}\rm
Suppose that the input-output signals \(u(t)\) and \(y(t)\) of the plant \eqref{dyn:LTI} can be measured. Without any prior knowledge of $A, B$ and $C$ beyond the
dimension $n$, design a dynamic output-feedback controller in the form of
\begin{equation}
\label{controller}
\dot \varrho = A_c \varrho + B_{c} \begmat{u \\y},~u = -K_c \varrho,
\end{equation}
with the controller state $\varrho \in \mathbb R^{n_\varrho}$, constant matrices $A_c \in \mathbb R^{n_\varrho \times n_\varrho}$ and $B_c \in \mathbb R^{n_\varrho \times (p+m)}$, and $K_c \in \mathbb R^{m \times n_\varrho}$ derived from the measured input-output data, such that
\begin{enumerate}
\item the closed-loop system consisting of  \eqref{dyn:LTI} and \eqref{controller} is stable; 
\item there exists a matrix $T_c \in \rea^{n \times n_\varrho}$ satisfying
\begin{equation}
\begin{aligned}
\lim_{t\to\infty} \| x(t) - T_c \varrho(t) \|  =0, ~{\rm and}~
K_c = K^* T_c,
\end{aligned}
\end{equation}
where $K^*$ is the optimal state-feedback gain of \eqref{form}.
\end{enumerate}
\end{problem}

\begin{remark}\rm
Since the system state is not accessible, a state-feedback controller for \eqref{form} cannot be directly implemented or learned.
To address this issue, we introduce an auxiliary higher-dimensional system with state $\varrho$, driven by the measured input-output signals.
Rather than directly computing the optimal state-feedback gain $K^*$,
the objective is to identify the optimal controller in the \(\varrho\)-coordinate, thereby being implementable. Consequently, the problem can be viewed as an input-output data-driven ADP problem.
\end{remark}

\subsection{Model-based Value Iteration}

We briefly recall a model-based VI algorithm for solving the ARE \eqref{ARE}, which is instrumental in our proposed approach.

\begin{algorithm}[!t]
\caption{Model-Based Value Iteration for LQR \cite{bian2016value}}
\label{alg:VI:modelbased}
\begin{algorithmic}[1]
\STATE \textbf{Input:} An initial matrix $P_0 \in \mathbb S_{++}^n$,  a tolerance  $\delta > 0$, $i = 0$ and $j = 0$
\STATE \textbf{Output:} $P_i$ as an approximation to $P^\ast$
\LOOP
\STATE $\tilde P_{i + 1} \!=\! P_i + \gamma_i\!\left(A^\top\! P_i \!+\! P_i A \!-\! P_i B R^{-1}\! B^\top\! P_i \!+\! C^\top \! Q C \right)$
\IF{$\tilde P_{i+1} \notin \mathbb{B}_j$}
\STATE $P_{i+1} = P_0$,~{\rm and} $j \leftarrow j+1$
\ELSIF{$\| \tilde P_{i+1} - P_i\|/\gamma_i \le \delta$}
\STATE \textbf{return} $P_i$ as an approximation to $P^\ast$
\ELSE
\STATE $P_{i+1} = \tilde P_{i+1}$
\ENDIF
\STATE $i \leftarrow i+1$
\ENDLOOP
\end{algorithmic}
\end{algorithm}

We define a stochastic approximation stepsize sequence $\{\gamma_i\}_{i = 0}^\infty$ satisfying
\begin{equation}
\label{VI:stepsize}
\gamma_i > 0, 
~\sum_{i = 0}^\infty \gamma_i= \infty, 
~\sum_{i = 0}^\infty \gamma_i^2 < \infty.
\end{equation}
To ensure the boundedness of the iterates, we introduce a collection of bounded subsets $\{\mathbb B_j\}_{j = 0}^\infty$ of $\mathbb S_+^n$ with nonempty interiors such that 
\begin{equation*}
\mathbb B_j \subset \mathbb B_{j+1},
~j \in \mathbb N,
~\lim_{j \to \infty} \mathbb B_j = \mathbb S_+^n.
\end{equation*}
Let $\delta$ be a prescribed tolerance.

Following \cite{bian2016value}, we summarize the VI in Algorithm \ref{alg:VI:modelbased}. Detailed convergence analysis and extensions to model-free settings can also be found therein.

%%%%
%%%%-----------------
\section{Proposed Framework}
\label{sec:3}

In this section, we propose an approach to solve Problem \ref{prob}. The main idea is to use a filter to get a canonical non-minimal realization of the plant \eqref{dyn:LTI}, and this has recently been shown effective in addressing offline direct data-driven control problems \cite{bosso2025data, bosso2026data, gao2025input}.

%%%%
%%%%-----------------
\subsection{Canonical Non-minimal Realization}

We first recall the definition of canonical non-minimal realization.

\begin{definition}[\!\!\cite{bosso2025data}]\rm
\label{def:real}
A non-minimal realization 
\begin{equation}
\label{dyn:real}
\dot \eta = (F + LH) \eta + Gu, ~ y = H \eta, ~\eta(0) = \eta_0
\end{equation}	
of system \eqref{dyn:LTI} with $\eta \in \rea^{n_\eta}$ and $n_\eta > n$
is canonical, if $F \in \mathbb R^{n_\eta \times n_{\eta}}$ is Hurwitz, and there exists  a full-row rank matrix $\Pi \in \mathbb R^{n \times n_\eta}$ such that
\begin{equation}
\label{eq:real}
\Pi (F + LH) = A \Pi, 
~\Pi G = B, ~{\rm and}
~H = C\Pi.
\end{equation}
\end{definition}

From \cite{bosso2025data}, the pair $(F+LH, G)$ is stabilizable, and the pair $(H, F+LH)$ is detectable.
Moreover, we have
\begin{equation}
\label{Real:state}
\Pi \eta_0 = x_0 \quad \implies \quad  \Pi \eta(t) \equiv x(t) , \; \forall t \ge 0.
\end{equation}

Suppose that a canonical non-minimal realization of \eqref{dyn:LTI}, given by \eqref{dyn:real}, is available, and the input-state trajectory of \eqref{dyn:real} is accessible. By endowing \eqref{dyn:real} with the same cost as in \eqref{form}, we derive an LQR problem as
\begin{equation}
\begin{aligned}
\label{reform:lqr}
\min~\int_0^\infty (y^\top Q y + u^\top R u) \mathrm{d} t,
~{\rm subject~to}~\eqref{dyn:real}.
\end{aligned}
\end{equation}

The following lemma clarifies the relationship between the optimal state-feedback controllers for \eqref{form} and \eqref{reform:lqr}.

\begin{lemma}\rm
\label{lem:real:opt}
Let $u = -K^* x$ be the optimal controller for \eqref{form}, and \eqref{dyn:real} be a canonical non-minimal realization of \eqref{dyn:LTI}.
Then, \eqref{reform:lqr} admits the optimal control law $u = -K^* \Pi \eta$, which  recovers the optimal controller for \eqref{form}, provided that $\Pi \eta_0 = x_0$. 
\end{lemma}

\begin{proof}
We begin by establishing the solvability of \eqref{reform:lqr}.
Note that $(F + LH, G)$ is stabilizable and $(H, F + LH)$ is detectable. Then, \eqref{reform:lqr} is solvable if $(\sqrt{Q}H, F + LH)$ is detectable. We show this by contradiction. If $(\sqrt{Q}H, F + LH)$ is not detectable, there exists $\lambda \in \mathbb C_+$ and $v \not= 0$ such that
\begin{equation*}
(F + LH)v = \lambda v, ~{\rm and}~\sqrt{Q}H v = 0.
\end{equation*}

By \eqref{eq:real}, $A \Pi v = \lambda \Pi v$ and $\sqrt{Q} C \Pi v = 0$. 
%In other words, $\lambda$ is an eigenvalue of $A$ associated with the eigenvector of $\Pi v$, or $\Pi v = 0$ otherwise.
Since $(\sqrt{Q}C, A)$ is observable, 
$\Pi v = 0$ and $C \Pi v = Hv = 0$.
Note that $(H, F + LH)$ is detectable. Hence, $(F + LH)v = \lambda v$ and $Hv = 0$ imply $v = 0$,  which contradicts with $v \not= 0$.
Therefore, $(\sqrt{Q}H, F + LH)$ is detectable, and \eqref{reform:lqr} is solvable.

Let $P^* \in \mathbb S_+^n$ be the solution to \eqref{ARE}.
It is clear that
\begin{equation*}
\begin{aligned}
\Pi^\top A^\top P \Pi &+ \Pi^\top  PA \Pi \\
&- \Pi^\top  PBR^{-1} B^\top P \Pi + \Pi^\top C^\top Q C \Pi = 0.
\end{aligned}
\end{equation*}
Recalling \eqref{eq:real} yields
\begin{equation*}
\begin{aligned}
(F &+ LH)^\top \Pi^\top P \Pi + \Pi^\top  P \Pi (F + LH) \\
&- \Pi^\top  P\Pi G R^{-1} G^\top \Pi^\top P \Pi +  H^\top Q  H = 0.
\end{aligned}
\end{equation*}

Thus, $P_\eta^* = \Pi^\top P^* \Pi \in \mathbb S_+^{n_{\eta}}$ is the unique solution to 
\begin{equation*}
\begin{aligned}
(F\!+\! LH)^\top P_\eta \!+\! P_\eta (F \!+\! LH) 
\!-\! P_\eta G R^{-1} G^\top P_\eta \!+\! H^\top Q H \!=\! 0.
\end{aligned}
\end{equation*}
Furthermore, the optimal controller for \eqref{reform:lqr} is given by
$$u = -  R^{-1} G^\top P_\eta^* \eta =  - R^{-1} B^\top P^* \Pi \eta = - K^* \Pi \eta.$$

By \eqref{Real:state}, $x(t) \equiv \Pi \eta(t)$ whenever $\Pi \eta_0 = x_0$. Since $u = -K^*x$ is the optimal controller of \eqref{form}, it follows directly that $u = -K^* \Pi \eta$ also solves \eqref{form}.
\end{proof}

Lemma \ref{lem:real:opt} indicates that, once a canonical non-minimal realization of \eqref{dyn:LTI} is available,
Problem \ref{prob} can be reformulated as an LQR problem associated with this realization.
Inspired by the observation, the framework for solving Problem \ref{prob} is illustrated in Fig. \ref{fig:frame}.
Specifically, the plant \eqref{dyn:LTI} generates input-output data, from which a canonical non-minimal realization \eqref{dyn:real} is constructed. The realization produces the non-minimal state trajectory, based on which a model-free algorithm will be studied to solve \eqref{reform:lqr} in the sequel.

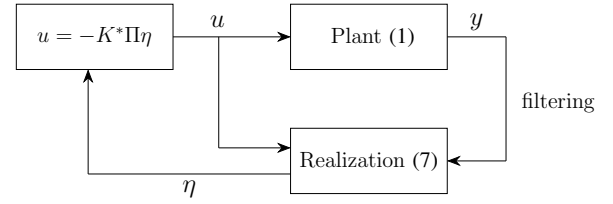
\begin{figure}[!ht]
\centering
\resizebox{.43\textwidth}{!}{%
\begin{circuitikz}
\tikzstyle{every node}=[font=\fontsize{18.2pt}{23.7pt}\selectfont]

\draw  (7.375,15.375) rectangle (10.375,14.125);
\draw  (7.375,13) rectangle (10.375,11.75);
\draw  (2.125,15.375) rectangle (5.125,14.125);

\draw [-{Stealth[scale=1.5]}, ] (5.125,14.75) -- (7.375,14.75)
node[pos=0.38, above=2pt, font=\fontsize{14pt}{14pt}] {$u$};

\draw [-{Stealth[scale=1.5]}, ] (11.5,12.375) -- (10.375,12.375);
\draw [-{Stealth[scale=1.5]}, ] (6,12.625) -- (7.375,12.625);
\draw [-{Stealth[scale=1.5]}, ] (3.5,12.125) -- (3.5,14.125);

\draw [short] (6,14.75) -- (6,12.625);
\draw [short] (3.5,12.125) -- (7.375,12.125)
node[midway, below, font=\fontsize{14pt}{14pt}] {$\eta$};

\draw [short] (11.5,12.375) -- (11.5,14.75);
\draw [short] (11.5,14.75) -- (10.375,14.75)
node[midway, above, font=\fontsize{14pt}{14pt}] {$y$};

\node [font=\fontsize{12pt}{23.7pt}\selectfont, fill={rgb,255:red,255; green,255; blue,255}, fill opacity=1, text opacity=1, inner xsep=0.080cm, inner ysep=0.085cm, rounded corners=0.020cm] at (3.625,14.75) {$u=-K^* \Pi \eta$};

\node [font=\fontfamily{cmr}\fontsize{12pt}{23.7pt}\selectfont, fill={rgb,255:red,255; green,255; blue,255}, fill opacity=1, text opacity=1, inner xsep=0.080cm, inner ysep=0.085cm, rounded corners=0.020cm] at (8.95,14.75) {Plant \eqref{dyn:LTI}};

\node [font=\fontfamily{cmr}\fontsize{12pt}{23.7pt}\selectfont, fill={rgb,255:red,255; green,255; blue,255}, fill opacity=1, text opacity=1, inner xsep=0.080cm, inner ysep=0.085cm, rounded corners=0.020cm] at (8.875,12.375) {Realization \eqref{dyn:real}};

\node [font=\fontfamily{cmr}\fontsize{12pt}{23.7pt}\selectfont, rotate around={-360:(0,0)}, fill={rgb,255:red,255; green,255; blue,255}, fill opacity=1, text opacity=1, inner xsep=0.080cm, inner ysep=0.085cm, rounded corners=0.020cm] at (12.5,13.5) {filtering};

\end{circuitikz}
}%
\caption{The proposed framework for solving Problem \ref{prob}.}
\label{fig:frame}
\end{figure}

%%%%
%%%%-----------------
\subsection{Kreisselmeier's Adaptive Filter}

Inspired by \cite{kreisselmeier2003generation, gao2025input}, we construct a canonical non-minimal realization of \eqref{dyn:LTI} using Kreisselmeier's adaptive filter, which is given by the following dynamical
system
\begin{equation}
\label{dyn:filter:Z}
\dot Z = MZ + \begmat{y^\top \otimes I_n \quad u^\top \otimes I_n}, ~Z(0) = Z_0
\end{equation}
where $Z \in \mathbb R^{n \times \tilde n_z}$ is the filter state, $\tilde n_z = (p + m)n$, and $M \in \mathbb R^{n \times n}$ is Hurwitz.

% \blue{We have the following key relation.}

\begin{lemma}[\!\!\cite{gao2025input}]
\label{lem:filter}\rm
Consider the system \eqref{dyn:LTI} and the filter \eqref{dyn:filter:Z}.
Suppose that $M$ is a Hurwitz matrix with distinct eigenvalues. Then, there exist a non-singular $T \in \rea^{n \times n}$ and a column vector $\theta \in \mathbb R^{\tilde n_z}$
such that
\begin{equation}
\label{state:sys:filter}
Tx(t) = Z(t) \theta + e^{Mt}(Tx_0 - Z_0 \theta).
\end{equation}
In addition, by defining $z := {\rm vec}(Z) \in \rea^{n_z}$ with $n_z := n^2 (p+m)$, the filter \eqref{dyn:filter:Z} can be written as 
\begin{equation}
\label{dyn:filter}
\dot z = A_z z + B_\xi u + L_z y,
\end{equation}
where $A_z = I_{\tilde n_z} \otimes M$,
\begin{equation*}
B_\xi = \begin{bmatrix}
0_{pn^2 \times m} \\
I_m \otimes {\rm vec}(I_n)
\end{bmatrix}, ~
L_z = \begin{bmatrix}
I_p \otimes {\rm vec}(I_n) \\
0_{mn^2 \times p} 
\end{bmatrix}.
\end{equation*}
\end{lemma}

With the help of \eqref{dyn:filter}, we are ready to introduce a canonical non-minimal realization of \eqref{dyn:LTI} in the sense of Definition \ref{def:real}.

\begin{lemma}[\!\!\cite{gao2025input}]
\label{lem:cano:real}\rm
The system described by
\begin{equation}
\label{dyn:real:filter}
\dot \xi = A_\xi  \xi + B_\xi u, ~ y = C_\xi  \xi, ~\xi(0) = \xi_0,
\end{equation}
is a canonical non-minimal realization of the plant \eqref{dyn:LTI},
where $\xi \in \mathbb R^{n_\xi}$, $n_\xi = n^2(p +m)$, $(\theta^\top \otimes T^{-1}) \xi_0 = x_0$, and
\begin{equation}
\label{eq:xi:z}
A_\xi := A_z + L_z C_\xi, ~
C_\xi := C(\theta^\top \otimes T^{-1}).
\end{equation}
\end{lemma}

\begin{remark}\rm
\label{rmk:observer}
It is straightforward to see that if $z_0 = \xi_0$, then the state equation \eqref{dyn:real:filter} coincides with \eqref{dyn:filter} satisfying $z(t) \equiv \xi(t), \forall t \ge 0$.
Moreover, the filter \eqref{dyn:filter} can be viewed as an exponentially convergent observer to \eqref{dyn:real:filter},
i.e., $\lim_{t \to \infty} \Vert z(t) - \xi(t)\Vert = 0$.
\end{remark}
% To see this, we define ${\tilde e}(t) := z(t) - \xi(t)$. It follows that
% $$
% \dot {\tilde e} = (I_{\tilde n_z} \otimes M) {\tilde e} + L_z (C x - C_\xi \xi).
% $$
% Note that $C x = C_\xi \xi$. Thus, $\tilde e(t)$ converges to $0$ with an exponential rate since $M$ is Hurwitz.

We equip system \eqref{dyn:real:filter} with the cost in \eqref{form}, and derive an LQR problem as
\begin{equation}
\begin{aligned}
\label{reform:filter}
\min~\int_0^\infty (y^\top Q y + u^\top R u) \mathrm{d} t, 
~{\rm subject~to}~\eqref{dyn:real:filter}.
\end{aligned}
\end{equation}

By Lemmas \ref{lem:real:opt} and \ref{lem:cano:real}, \eqref{reform:filter} admits the
optimal control law  $u(t) = - K^*(\theta^\top \otimes T^{-1}) \xi(t)$, which also solves problem \eqref{form}, provided that the initial conditions satisfy
\begin{equation}
\label{eq:IC}
(\theta^\top \otimes T^{-1}) \xi_0 = x_0 .
\end{equation}
Here,
$K^*$ is the optimal state-feedback control gain of \eqref{form}. Note that, however, this condition cannot be directly verified in practice, since $x_0$ is generally unknown.

%%%%%
%%%%%----------------
\subsection{Output-Feedback Controller}

The previous subsection shows that the canonical non-minimal realization \eqref{dyn:real:filter} leads to an LQR reformulation of \eqref{form}, provided that the initial matching condition \eqref{eq:IC} holds.
As discussed above, however, this condition cannot be enforced in
practice, since both the initial state $x_0$ and the exact plant model are unknown in the present model-free setting.

Consequently, the non-minimal state $\xi(t)$ cannot be used directly for feedback. Instead, we reconstruct the required trajectory through the filter \eqref{dyn:filter}, or equivalently \eqref{dyn:filter:Z}. To account for the resulting transient behavior, we introduce the following higher-dimensional realization.

\begin{lemma}\rm
Suppose that $M$ is Hurwitz with distinct real eigenvalues.
Given the plant \eqref{dyn:LTI}, define the system
\begin{equation}
\begin{aligned}
\label{dyn:overall}
\dot \zeta &= 
\underbrace{
\begin{bmatrix}\!
A_\xi \!\!&L_z CT^{-1} \Gamma \\
0_{n \times n_z} \!\!&\Lambda_M 
\!\end{bmatrix}
}_{A_\zeta} \zeta +
\underbrace{ 
\begin{bmatrix}
B_\xi \\
0_{n \times m} 
\end{bmatrix}
}_{B_\zeta} u, 
~ \zeta(0) = \begmat{0_{n_z}\\ 1_n}
\\
y_\zeta &= 
\underbrace{
\begin{bmatrix}
C_\xi~~~
CT^{-1} \Gamma
\end{bmatrix}
}_{C_{\zeta}}
\zeta,
\end{aligned}
\end{equation}
where $\zeta \in \mathbb R^{n_\zeta}$, $n_\zeta := n_z + n$,  $\Lambda_M = {\rm diag}\{\lambda^M_i\}$, $\lambda^M_i$ is the $i$th eigenvalue of $M$, and $\Gamma$ is a constant matrix. Then, for any $u: \rea_{\ge 0} \to \rea^m$, the outputs $y_\zeta$ and $y$ from \eqref{dyn:overall} and \eqref{dyn:LTI}, respectively, satisfy
$$
y_\zeta(t) \equiv y(t), ~\forall t \ge 0.
$$
\end{lemma}
\begin{proof}
Define the transverse coordinate $\epsilon \in \mathbb R^n$ \cite{andrieu2016transverse} as
\begin{equation}
\label{transverse}
\epsilon(t) := Tx(t) - Z(t) \theta.
\end{equation}

Let $z_0 = 0_{n_z}$. Then, $\epsilon(0) = \epsilon_0 = Tx_0$. Recalling \eqref{state:sys:filter} gives 
$\epsilon(t) = \exp(Mt) \epsilon_0$, and thus, $\dot \epsilon = M \epsilon$.
Since $M$ has distinct real eigenvalues $\lambda_i^M$ ($i=1,\ldots, n$), there exists an invertible matrix $\tilde T \in \mathbb R^{n \times n}$ such that $M = \tilde T \Lambda_M \tilde T^{-1}$, where $\Lambda_M = {\rm diag}\{\lambda^M_i\}$. 
Therefore, we obtain
$$\epsilon(t) = \tilde T e^{\Lambda_M t} \tilde T^{-1} \epsilon_0 =: \Gamma \varepsilon(t)$$
for some $\Gamma \in \mathbb R^{n \times n}$ with
\begin{equation}
\label{dyn:error:reform}
\dot \varepsilon(t) = \Lambda_M \varepsilon(t), ~\varepsilon(0) = 1_n.
\end{equation}

In light of \eqref{dyn:filter}, we have 
\begin{equation*}
\begin{aligned}
\label{dyn:filter:error}
\dot z 
% % &= A_z z + B_z u + L_z y \\
&= A_\xi z + B_\xi u + L_z(C x - C_\xi z) \\
&= A_\xi z + B_\xi u + L_z C T^{-1} \epsilon.
\end{aligned}
\end{equation*}
In addition, it follows from \eqref{transverse} that $ y = Cx = C_{\xi} z + CT^{-1}\epsilon$.
By defining the extended state $\zeta := [z^\top, \varepsilon^\top]^\top$, we have its dynamics given by \eqref{dyn:overall}. It is clear that $y_\zeta(t) \equiv y(t)$.
\end{proof}

We make the following two important observations.
\begin{enumerate}
\item[(i)] The system \eqref{dyn:overall} shares the same input-output behavior as the plant \eqref{dyn:LTI}.

\item[(ii)] The state trajectory $\zeta(t)$ is available, as it is generated by the filters \eqref{dyn:filter:Z} and \eqref{dyn:error:reform}.
\end{enumerate}

Furthermore, we entail \eqref{dyn:overall} with the same
cost as in \eqref{form}, and then derive an LQR problem as
\begin{equation}
\begin{aligned}
\label{reform}
\min~\int_0^\infty (y^\top Q y + u^\top R u) \mathrm{d} t, 
~{\rm subject~to}~\eqref{dyn:overall}.
\end{aligned}
\end{equation}
The following result holds.

\begin{proposition}\rm
\label{prop:opt:contrls}
Let $u = -K^* x$ be the optimal state-feedback controller for \eqref{dyn:LTI}. Suppose that $M$ is Hurwitz with distinct real eigenvalues. Then, \eqref{reform} admits an optimal controller in the form of  $u = - K_z^* z - K_{\varepsilon}^* \varepsilon$, in which the optimal gain satifies $K_z^* = K^*(\theta^\top \otimes T^{-1})$.
\end{proposition}

\begin{proof}
We first show the stabilizability of $(A_\zeta, B_\zeta)$.
Clearly, the uncontrollable modes of \eqref{dyn:overall} are contained in the eigenvalues of $\Lambda_M$ and the uncontrollable modes of $(A_\xi, B_\xi)$. Since $\Lambda_M$ is Hurwitz and $(A_\xi, B_\xi)$ stabilizable, $(A_\zeta, B_\zeta)$ is stabilizable.
Similarly, both $(C_\zeta, A_\zeta)$ and $(\sqrt{Q}C_\zeta, A_\zeta)$ are detectable.
Thus, \eqref{reform} is solvable by a state-feedback controller.

Let $P^*_\zeta \in \mathbb S_+^{n_\zeta}$ be the unique solution to
\begin{equation}
\label{ARE:zeta}
A_\zeta^\top P_\zeta + P_\zeta A_\zeta - P_\zeta B_\zeta R^{-1} B_\zeta^\top P_\zeta + C_\zeta^\top Q C_\zeta = 0.
\end{equation}
Note that $P^*_\zeta$ admits the block partition
$$
P_\zeta^* =
\begin{bmatrix}
P_{zz}^* & P_{z\varepsilon}^* \\
P_{\varepsilon z}^* & P_{\varepsilon\varepsilon}^*
\end{bmatrix}.
$$
As a result, the optimal control law for \eqref{reform} is 
\begin{equation*}
u = - K_\zeta^* \zeta = - R^{-1}B_\zeta^\top P^*_\zeta \zeta =: - K_z^* z - K_\varepsilon^* \varepsilon,
\end{equation*}
where $K_z^* :=  R^{-1}B_{\xi}^\top P_{zz}^*$, and
$K_\varepsilon^* := R^{-1}B_{\xi}^\top P_{z \epsilon}^*$.  Plugging \eqref{dyn:overall} into \eqref{ARE:zeta}, we derive
\begin{equation*}
\begin{aligned}
&A_{\xi}^\top P^*_{zz} + P^*_{zz} A_{\xi} - P^*_{zz} B_{\xi} R^{-1} B_{\xi}^\top P^*_{zz} + C_{\xi}^\top Q C_{\xi} = 0. 
\end{aligned}
\end{equation*}

Then, it is straightforward to verify that $u = - K_z^* \xi$ is the optimal controller for problem \eqref{reform:filter}.
Since  \eqref{dyn:real:filter} is a canonical non-minimal realization of \eqref{dyn:LTI},
it holds that $K_z^*=K^*(\theta^\top \otimes T^{-1})$ by Lemma \ref{lem:real:opt}.
\end{proof}

\begin{remark}\rm
The overall dynamics \eqref{dyn:overall} extends the state of \eqref{dyn:real:filter} by introducing the auxiliary variable $\varepsilon$. Proposition \ref{prop:opt:contrls} shows that 
such an augmentation does not alter the optimal feedback gain associated with the realization \eqref{dyn:real:filter}. In particular, $K_z^*$ coincides exactly with the optimal gain  of \eqref{reform:filter}.
\end{remark}

\begin{remark}\rm
\label{rmk:contrl}
The trajectories of $(u,\zeta)$ for system \eqref{dyn:overall} are available through the filtering dynamics \eqref{dyn:filter:Z} and \eqref{dyn:error:reform}. Under suitable excitation conditions, these data can be used to estimate the optimal control gains, denoted by $(\hat K_z^*, \hat K_\varepsilon^*)$. By Proposition \ref{prop:opt:contrls}, the component $\hat K_z^*$ serves as an estimate of the optimal control gain associated with the non-minimal realization \eqref{dyn:real:filter}. Since $z(t)$ converges exponentially to $\xi(t)$ as $t \to \infty$, the resulting control law $u=-\hat K_z^* z$ provides an exponentially convergent estimate of the optimal control law. This forms the basic idea for addressing Problem \ref{prob}.
\end{remark}

%%%%
%%%%-----------------
\section{Data-Driven Algorithm}

In this section, we develop an iterative algorithm, using only the input-output trajectory data, to solve Problem \ref{prob} by leveraging the optimal control of \eqref{reform}. Note that $A_\zeta$ and $C_\zeta$ are unknown, whereas the input matrix $B_\zeta$ is known in our design. Since the signals $(u,\zeta)$ are available, this makes it possible to design data-driven optimal controller for \eqref{reform} using ADP.

The optimal control law for \eqref{reform} is given by
\begin{equation}
\label{reform:opt:gain}
u(t) = - K_\zeta^* \zeta(t),
\end{equation}
where $K_\zeta^* = R^{-1} B_\zeta^\top P_\zeta^*$, and $P_\zeta^* \in \mathbb S_+^{n_\zeta}$ is the unique solution to the ARE
\begin{equation}
\label{ARE:reform}
A_\zeta^\top P_\zeta + P_\zeta A_\zeta - P_\zeta B_\zeta R^{-1} B_\zeta^\top P_\zeta + C_\zeta^\top Q C_\zeta = 0.
\end{equation}

We are now ready to solve Problem \ref{prob} by combining the fundamental idea in Section \ref{sec:3} with the VI algorithm.

To iteratively solve the ARE \eqref{ARE:reform} and compute the optimal control gain $K^*_\zeta$ for \eqref{reform}, we use the stepsize sequence $\{\gamma_i\}$ defined in \eqref{VI:stepsize}, and introduce a collection of nonempty bounded subsets $\{\mathbb{\tilde B}_j\}$ satisfying
$\mathbb{\tilde B}_j \subset \mathbb{\tilde B}_{j+1}, 
j \in \mathbb N$ and $\lim_{j \to \infty} \mathbb{\tilde B}_j = \mathbb S_+^{n_\zeta}$.
Following Algorithm \ref{alg:VI:modelbased}, we update $P_{\zeta, i}$ via
\begin{equation}
\begin{aligned}
\label{VI:P1}
\tilde P_{\zeta, i+1} =&~ P_{\zeta, i} + \gamma_i(A_\zeta^\top P_{\zeta, i} + P_{\zeta, i} A_\zeta \\
&- P_{\zeta, i} B_\zeta R^{-1} B_\zeta^\top P_{\zeta, i} + C_\zeta^\top Q C_\zeta).
\end{aligned}
\end{equation}

Let $K_{\zeta, i} = R^{-1} B_\zeta^\top P_{\zeta, i}$.
Denote by 
$$
O_{\zeta, i} := A_{\zeta}^\top P_{\zeta, i} + P_{\zeta, i} A_{\zeta} + C_\zeta^\top Q C_\zeta.
$$
Then, \eqref{VI:P1} can be rewritten as
\begin{equation}
\begin{aligned}
\label{VI:P2}
\tilde P_{\zeta, i+1} = P_{\zeta, i} + \gamma_i(O_{\zeta, i} 
- K_{\zeta, i}^\top R K_{\zeta, i}).
\end{aligned}
\end{equation}
Therefore, to implement the VI algorithm, it is essential to compute $O_{\zeta, i}$ using the trajectories of $u(t)$, $\zeta(t)$ and $y(t)$.

Taking the derivative of $\zeta^\top P_{\zeta, i} \zeta$ along \eqref{dyn:overall}, we have
\begin{equation*}
\begin{aligned}
\frac {\mathrm{d}}{\mathrm{d}t} (\zeta^\top P_{\zeta, i} \zeta) 
= \zeta^\top O_{\zeta, i} \zeta + 2u^\top B_\zeta^\top P_{\zeta, i}\zeta - y^\top Q y.
\end{aligned}
\end{equation*}

Integrating the above equation over $[t_j, t_{j+1}]$  and repeating this procedure on multiple time intervals, we obtain a set of linear equations represented in the following matrix form
\begin{equation}
\label{VI:eq2}
I_{\zeta\zeta} \cdot {\rm vec}(O_{\zeta, i}) = \Xi_i,
\end{equation}
where $l$ is a positive integer,  $0 \le t_0 < t_1 < \dots < t_l$, and
$$
\begin{aligned}
\Xi_i & = \big[\delta_{\zeta\zeta} - 2 I_{\zeta u} \cdot (I_{n_\zeta} \otimes  B_\zeta^\top)\big] {\rm vec}(P_{\zeta, i}) + I_{yy} \cdot {\rm vec}(Q),
\\
\delta_{\zeta \zeta} &= \big[\zeta \otimes \zeta|_{t_0}^{t_1}, ~\zeta \otimes \zeta|_{t_1}^{t_2}, ~\dots,~ \zeta \otimes \zeta|_{t_{l-1}}^{t_l} \big]^\top, \\
I_{\zeta \zeta}
&= \begmat{\int_{t_0}^{t_1} \zeta \otimes \zeta \mathrm{d}t & \int_{t_1}^{t_2} \zeta \otimes \zeta \mathrm{d}t & \dots & \int_{t_{l-1}}^{t_l} \zeta \otimes \zeta \mathrm{d}t  }^\top,
\\
I_{\zeta u}&= 
\begmat{
\int_{t_0}^{t_1} \zeta \otimes u \mathrm{d}t & 
\int_{t_1}^{t_2} \zeta \otimes u \mathrm{d}t & \dots & \int_{t_{l-1}}^{t_l} \zeta \otimes u \mathrm{d}t 
}^\top \!, 
\\
I_{yy}
&= 
\begmat{\int_{t_0}^{t_1} y \otimes y \mathrm{d}t& \int_{t_1}^{t_2} y \otimes y \mathrm{d}t & \dots & \int_{t_{l-1}}^{t_l} y \otimes y \mathrm{d}t}^\top \!.
\end{aligned}
$$

Hence, the symmetric matrix $O_{\zeta,i}$ can be estimated from trajectory data by solving  \eqref{VI:eq2}, which has a unique solution if 
\begin{equation}
\label{ass:rank}
{\rm rank}\;\{I_{\zeta \zeta}\} = {{n_\zeta(n_\zeta + 1)} \over 2}.
\end{equation}
This rank condition depends on the trajectory of $\zeta(t)$, determined by the input $u(t)$.
In practice, the excitation condition can be ensured by injecting suitable probing signals, such as random noise \cite{al2007model}, exponentially decreasing probing
noise \cite{vamvoudakis2011multi} and sinusoidal signals  \cite{jiang2012computational}.

We are now ready to present our approach for solving Problem \ref{prob}, which is summarized in Algorithm \ref{alg:VI:reform}. Its properties are given as follows.

\begin{algorithm}[tp]
\caption{Data-Driven Control for Solving Problem \ref{prob}}
\label{alg:VI:reform}
\begin{algorithmic}[1]
\STATE \textbf{Input:} Initialize $P_{\zeta, 0} \in \mathbb S_{++}^{n_\zeta}$, set $i = 0$, $j = 0$
and a tolerance $\delta > 0$, give the sets $\{\tilde{\mathbb B}_j\}$ and the stepsizes 
$\{\gamma_i\}$, and select a Hurwitz matrix $M$.
\STATE \textbf{Output:} The control law for Problem \ref{prob}.
\STATE Construct the filter
$$\dot Z = M Z + \begmat{y^\top \otimes I_n & u^\top \otimes I_n}, 
~Z_0 = 0_{n \times \tilde n_z}.$$
\STATE Compute $\Lambda_M$, and derive the dynamics 
$$\dot \varepsilon(t) = \Lambda_M \varepsilon(t), ~\varepsilon(0) = 1_n.$$
\STATE Apply a measurable locally essentially bounded input $u$ to the plant \eqref{dyn:LTI}, collect $u(t)$ and $y(t)$, and generate $\zeta(t)$ through \eqref{dyn:filter:Z} and \eqref{dyn:error:reform} until \eqref{ass:rank} is satisfied.
\LOOP
\STATE Solve $O_{\zeta, i}$ from \eqref{VI:eq2}
\STATE $K_{\zeta, i} = R^{-1} B_\zeta^\top P_{\zeta,i}$
\STATE $\tilde P_{\zeta, i+1} = P_{\zeta, i} + \gamma_i\!\big(O_{\zeta, i} - K_{\zeta, i}^\top R K_{\zeta, i} \big)$
\IF{$\tilde P_{\zeta, i+1} \notin \mathbb{\tilde B}_j$}
\STATE $P_{\zeta, i+1} = P_{\zeta, 0}$,~{\rm and} $j \leftarrow j+1$
\ELSIF{$\|\tilde P_{\zeta, i+1} - P_{\zeta, i}\|/\gamma_i \le \delta$}
\STATE \textbf{return} $P_{\zeta, i}$ and $K_{\zeta, i}$
\ELSE
\STATE $P_{\zeta, i+1} = \tilde P_{\zeta, i + 1}$
\ENDIF
\STATE $i \leftarrow i+1$
\ENDLOOP
\STATE Let $\hat K_{\zeta}^* := K_{\zeta, i}$ denote the estimate of the optimal gain $K_{\zeta}^*$. Partition $\hat K_{\zeta}^*$ as $\hat K_{\zeta}^* = [\hat K_{z}^*, \hat K^*_{\varepsilon}]$.
\STATE \textbf{Controller:} $u(t) = - \hat K^*_z {\rm vec}(Z(t))$.
\end{algorithmic}
\end{algorithm}

\begin{theorem}\rm
\label{thm:VI}
Consider the plant \eqref{dyn:LTI} with access only to the signals $(u,y)$, the filter \eqref{dyn:filter:Z}, where $M$ is Hurwitz with distinct real eigenvalues, and the dynamics \eqref{dyn:error:reform}. Suppose that the rank condition \eqref{ass:rank} holds for the collected data.
Then, the sequence $\{K_{\zeta,i}\}_{i=0}^{\infty}$, generated by Algorithm \ref{alg:VI:reform}, converges to the optimal state-feedback control gain $K_\zeta^*$ associated with \eqref{reform}.
Let $\hat K_{\zeta}^{*}$ be the estimate of $K_\zeta^*$, and partition 
$K_\zeta^* = [K_z^*, K_\varepsilon^*]$ and
$\hat K_\zeta^* = [\hat K_z^*, \hat K_\varepsilon^*]$. It holds that $K_z^* = K^*(\theta^\top \otimes T^{-1})$, where $K^*$ is the optimal state-feedback gain of \eqref{form}.
With the controller
\begin{equation}
\label{Contrl}
u(t) = - \hat K_{z}^* {\rm vec}(Z(t)),
\end{equation}
the resulting closed-loop system, consisting of \eqref{dyn:LTI}, \eqref{dyn:filter:Z} and \eqref{Contrl}, is exponentially stable.
\end{theorem}

\begin{proof}
With the filter \eqref{dyn:filter:Z}, the iterative procedure in Algorithm \ref{alg:VI:reform} for computing $P_{\zeta,i}$ can be interpreted as the VI applied to the extended dynamics \eqref{dyn:overall}. Hence, following arguments similar to those in \cite[Theorem 4.2]{bian2016value}, we obtain
\[
\lim_{i \to \infty} P_{\zeta,i} = P_{\zeta}^*, 
\qquad 
\lim_{i \to \infty} K_{\zeta,i} = K_{\zeta}^*,
\]
where $P_{\zeta}^* \in \mathbb S_+^{n_{\zeta}}$ solves \eqref{ARE:reform}, and $K_{\zeta}^*$ is the optimal control gain for \eqref{reform}.
By Proposition \ref{prop:opt:contrls}, we derive $K_z^* = K^*(\theta^\top \otimes T^{-1})$, where $K^*$ be the optimal state-feedback gain of \eqref{form}.

Since $K_\zeta^*$ is the optimal state-feedback gain for \eqref{reform},
the matrix $(A_\zeta - B_\zeta K_\zeta^*)$ is Hurwitz. 
It follows from \eqref{dyn:overall} that the eigenvalues of $A_\zeta - B_\zeta K_\zeta^*$ consist of those of 
$(A_\xi-B_\xi K_z^*)$ and $\Lambda_M$, which implies that $(A_\xi-B_\xi K_z^*)$ is Hurwitz.
Plugging \eqref{Contrl} into \eqref{dyn:filter}, we derive
$$\dot z = (A_\xi - B_\xi K_z^*) z + L_z CT^{-1} \Gamma \varepsilon.$$
Thus, $z(t)$ decays exponentially since $\varepsilon(t)$ does. 
By \eqref{state:sys:filter}, we further conclude that $x(t)$ exponentially converges to $0$.
\end{proof}

\begin{remark}\rm
Since $B_\zeta$ in \eqref{dyn:overall} is known, $K_{\zeta,i}$ can be computed directly after obtaining $O_{\zeta,i}$. Hence, unlike \cite{bian2016value}, we only compute $O_{\zeta,i}$ from \eqref{VI:eq2}, instead of solving for $O_{\zeta,i}$ and $K_{\zeta,i}$ simultaneously. This leads to a weaker rank condition than that in \cite{bian2016value} and reduces the per-iteration computational burden. A similar idea can be found in \cite{lin2025new}.
\end{remark}

\begin{remark}\rm
Theorem \ref{thm:VI} shows that the optimal controller of \eqref{form} can be derived using only input-output data, without requiring access to the system state. This is achieved by introducing the filter \eqref{dyn:filter:Z} together with the transverse  dynamics \eqref{dyn:error:reform}.
By Proposition \ref{prop:opt:contrls}, the optimal control gain of \eqref{reform} admits an explicit relation with that of \eqref{form}.
Therefore, once \eqref{reform} is solved by a model-free algorithm, Problem \ref{prob} can be addressed.
\end{remark}

%%%%
%%%%-----------------
\section{Illustrative Example}

In this section, we validate the proposed approach via simulations on an output-feedback normal-acceleration regulator for a linearized F-16 aircraft model.
The system state of interest is $x = [\alpha, q, \delta_e]^\top$, where $\alpha$, $q$, and $\delta_e$ denote the angle of attack, pitch rate, and elevator deflection angle, respectively.
The linearized model is given by \cite{stevens2015aircraft}
\begin{equation*}
\begin{aligned}
\dot x &=
{\setlength{\arraycolsep}{3pt} 
\begin{bmatrix}
-1.01887 &0.90506  &-0.00215 \\
0.82225  &-1.07741 &-0.17555 \\
0        &0        &-20.2   
\end{bmatrix}}x +  
{\setlength{\arraycolsep}{3pt} 
\begin{bmatrix}
0 \\
0 \\
20.2 
\end{bmatrix}}u, \\
y &= 
{\setlength{\arraycolsep}{4pt} 
\begin{bmatrix}
0       &57.2958       &0    
\end{bmatrix}}x,
\end{aligned}
\end{equation*}
where $u$ is the elevator control input, and $y$ denotes the pitch rate with the factor of 57.2958 converting angles from radians to degrees.
The weighting matrices in the cost \eqref{form} are set to $Q = 1$ and $R = 1$.
It is straightforward to verify that $(A, B)$ is controllable, and both $(C, A)$ and $(\sqrt{Q}C, A)$ are observable.
Moreover, the optimal state-feedback control gain is $K^* = [-5.4636, -49.779, 0.36616]$.

We choose a Hurwitz matrix $M$ with distinct real eigenvalues in \eqref{dyn:filter:Z}. 
Following \cite{gao2025input}, this yields
\begin{equation*}
\begin{aligned}
M \!=\! 
{\setlength{\arraycolsep}{3pt} 
\begin{bmatrix}
-1  &0   &0  \\
1   &-2  &0   \\
0   &2   &-3
\end{bmatrix}}\!,~
T \!=\! 
{\setlength{\arraycolsep}{3pt} 
\begin{bmatrix}
 423.92	   &9.7288	 &-0.13642 \\
-232.81	   &793.37   &-7.6176  \\
-155.30	   &-192.11	 &2.8659
\end{bmatrix}}\!, \\
\theta = [6.6833, 8.9277, -36.593, 
-2.7558, -153.87, 57.891]^\top.
\end{aligned}
\end{equation*}

We now validate the performance of Algorithm \ref{alg:VI:reform}.
The initial feedback gain is set to $K_{\zeta, 0} = 0_{1 \times n_\zeta}$.
To ensure the persistent excitation,
the initial control input is chosen as $u(t) = - K_{\zeta, 0} \zeta(t) + e(t)$, where $e(t) = \sum_{\kappa = 1}^{100}(20 + \iota_\kappa) \sin(\omega_\kappa t + \phi_\kappa)$ with $\iota_\kappa$, $\omega_\kappa$ and $\phi_\kappa$ randomly sampled from $[0,20]$, $[-500,500]$ and $[0,2\pi]$.
The input $u(t)$, output $y(t)$, and state $\zeta(t)$ are collected every $0.01\mathrm{s}$ over the interval $[0,5]$s.
We choose $\gamma_i = (1 + i)^{-1}$ and $\tilde B_j = 2^j \cdot 5 I_{n_\zeta}$.
We then verify the rank condition \eqref{ass:rank} and start the VI algorithm.

Fig.~\ref{Fig2} shows the trajectories of
$\|P_{\zeta,i}-P_\zeta^*\|_F/\|P_\zeta^*\|_F$ and
$\|K_{\zeta,i}-K_\zeta^*\|_F/\|K_\zeta^*\|_F$.
It is clear that $P_{\zeta, i}$
and $K_{\zeta, i}$ converge to their optimal values 
$P^*_\zeta$ and $K^*_\zeta$.

With the controller \eqref{Contrl}, Fig.~\ref{Fig3} shows the trajectories of $x(t)$ and $\log(\| Z(t)\|_F)$, demonstrating that the closed-loop system, consisting of \eqref{dyn:LTI}, \eqref{dyn:filter:Z} and \eqref{Contrl}, is stable.

Fig.~\ref{Fig4} compares the control input and output trajectories generated by the controller $u(t)=-K_z^{*}z(t)$ and the optimal state-feedback controller $u(t)=-K^{*}x(t)$, where both $K_z^{*}$ and $K^{*}$ are characterized in Theorem \ref{thm:VI}.
It is clear that the control input produced by the learned controller converges to that of the optimal state-feedback controller. 
This is consistent with the results established in Theorem \ref{thm:VI}.

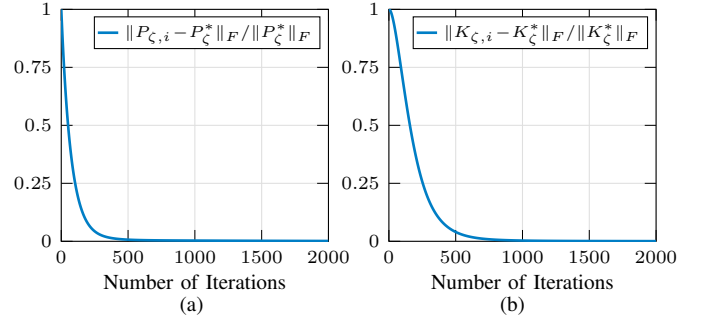
\begin{figure}[!t]
\centering

\begin{subfigure}{0.498\linewidth}
\centering
\begin{tikzpicture}
\begin{axis}[
width=1.16\linewidth,
height=4.65cm,
xmin=0,
xmax=2000,
ymin=0,
ymax=1,
xtick={0,500,1000,1500,2000},
ytick={0,0.25,0.5,0.75,1},
scaled x ticks=false,
xlabel={Number of Iterations},
xlabel style={yshift=0.8mm},
ylabel={},
grid=major,
major grid style={gray!25},
tick label style={font=\scriptsize},
label style={font=\footnotesize},
legend style={
at={(0.97,0.97)},
anchor=north east,
font=\fontsize{6.5}{6.5}\selectfont,
draw=black,
fill=white,
inner sep=0.8pt,
legend image post style={xscale=0.45}
},
xticklabel style={/pgf/number format/1000 sep={}},
axis line style={black},
tick style={black}
]
\addplot[color=cyan!60!blue,line width=1.0pt]
table[x=iter,y=P_error]{fig1_data.dat};

\legend{
$\|P_{\zeta,i}\!-\! P_{\zeta}^{*}\|_F/\|P_{\zeta}^{*}\|_F$}
\end{axis}
\end{tikzpicture}

\vspace{-2.6mm}
\makebox[1.12\linewidth][c]{\footnotesize (a)}
\phantomcaption
%\caption{}
\label{fig:P_error}
\end{subfigure}
\hspace{-3mm}
\begin{subfigure}{0.498\linewidth}
\centering
\begin{tikzpicture}
\begin{axis}[
width=1.16\linewidth,
height=4.65cm,
xmin=0,
xmax=2000,
ymin=0,
ymax=1,
xtick={0,500,1000,1500,2000},
ytick={0,0.25,0.5,0.75,1},
scaled x ticks=false,
xlabel={Number of Iterations},
xlabel style={yshift=0.8mm},
ylabel={},
grid=major,
major grid style={gray!25},
tick label style={font=\scriptsize},
label style={font=\footnotesize},
legend style={
at={(0.97,0.97)},
anchor=north east,
font=\fontsize{6.5}{6.5}\selectfont,
draw=black,
fill=white,
inner sep=0.8pt,
legend image post style={xscale=0.45}
},
xticklabel style={/pgf/number format/1000 sep={}},
axis line style={black}, tick style={black}]
\addplot[color=cyan!60!blue, line width=1.0pt]
table[x=iter, y=K_error]{fig1_data.dat};

\legend{$\|K_{\zeta,i}\!-\!K_{\zeta}^{*}\|_F/\|K_{\zeta}^{*}\|_F$}
\end{axis}
\end{tikzpicture}

\vspace{-2.6mm}
\makebox[1.08\linewidth][c]{\footnotesize (b)}
\phantomcaption
%\caption{}
\label{fig:K_error}
\end{subfigure}
\vspace{-4mm}
\caption{
Convergence of Algorithm~2.
(a) Trajectory of $\|P_{\zeta,i}-P_{\zeta}^{*}\|_F/\|P_{\zeta}^{*}\|_F$.
(b) Trajectory of $\|K_{\zeta,i}-K_{\zeta}^{*}\|_F/\|K_{\zeta}^{*}\|_F$.}
\label{Fig2}
\end{figure}

\begin{figure}[!t]
\centering

\begin{subfigure}{0.498\linewidth}
\centering
\begin{tikzpicture}
\begin{axis}[
width=1.16\linewidth,
height=4.65cm,
xmin=0,
xmax=28,
ymin=-4,
ymax=2,
xtick={0,7,14,21,28},
ytick={-4,-2,0,2},
scaled x ticks=false,
xlabel={Time (sec)},
xlabel style={yshift=0.8mm},
ylabel={},
grid=major,
major grid style={gray!25},
tick label style={font=\scriptsize},
label style={font=\footnotesize},
legend style={
at={(0.97,0.97)},
anchor=north east,
font=\fontsize{6.5}{6.5}\selectfont,
draw=black,
fill=white,
inner sep=0.8pt,
legend image post style={xscale=0.45}
},
xticklabel style={/pgf/number format/1000 sep={}},
axis line style={black},
tick style={black}
]
\addplot[color=cyan!60!blue,line width=1.0pt,mark=none]
table[x=x,y=y1]{fig2_left.dat};
\addplot[color=orange,line width=1.0pt,mark=none]
table[x=x,y=y2]{fig2_left.dat};
\addplot[color=orange!40!yellow,line width=1.0pt,mark=none]
table[x=x,y=y3]{fig2_left.dat};

\legend{$\alpha~{\rm (rad)}$, $q~{\rm (rad/s)}$, $\delta_e~{\rm (rad)}$}
\end{axis}
\end{tikzpicture}

\vspace{-2.6mm}
\makebox[1.12\linewidth][c]{\footnotesize (a) Trajectory of $x(t)$}
\phantomcaption
%\caption{}
\label{fig:x_decay}
\end{subfigure}
\hspace{-3mm}%
\begin{subfigure}{0.498\linewidth}
\centering
\begin{tikzpicture}
\begin{axis}[
width=1.16\linewidth,
height=4.65cm,
xmin=0,
xmax=28,
ymin=-10,
ymax=2,
xtick={0,7,14,21,28},
ytick={-10,-6,-2,2},
scaled x ticks=false,
xlabel={Time (sec)},
xlabel style={yshift=0.8mm},
ylabel={},
grid=major,
major grid style={gray!25},
tick label style={font=\scriptsize},
label style={font=\footnotesize},
legend style={
at={(0.97,0.97)},
anchor=north east,
font=\fontsize{6.5}{6.5}\selectfont,
draw=black,
fill=white,
inner sep=0.8pt,
legend image post style={xscale=0.45}
},
xticklabel style={/pgf/number format/1000 sep={}},
axis line style={black},
tick style={black}
]
\addplot[color=cyan!60!blue, line width=1.0pt]
table[x=x, y=y]{fig2_right.dat};

\legend{$\log(\|Z(t)\|_F)$}
\end{axis}
\end{tikzpicture}

\vspace{-2.6mm}
\makebox[1.08\linewidth][c]{\footnotesize (b) Trajectory of $\log(\|Z\|_F)$}
\phantomcaption
%\caption{}
\label{fig:Z_decay}
\end{subfigure}
\vspace{-4mm}
\caption{Convergence of the states $x$ and $Z$.}
\label{Fig3}
\end{figure}
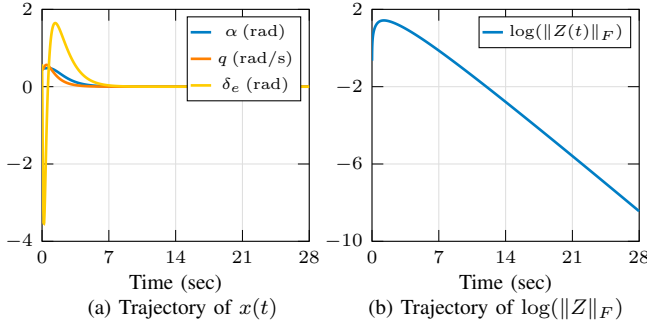

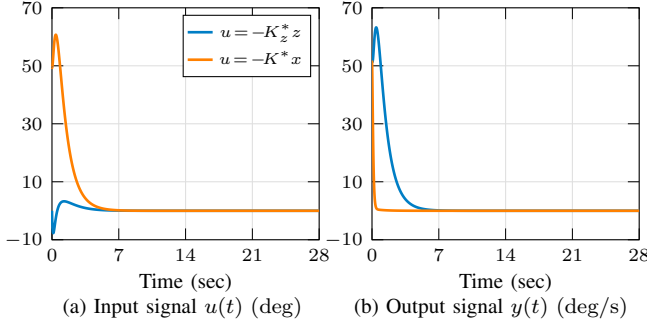
\begin{figure}[!t]
\centering

\begin{subfigure}{0.498\linewidth}
\centering
\begin{tikzpicture}
\begin{axis}[
width=1.16\linewidth,
height=4.65cm,
xmin=0,
xmax=28,
ymin=-10,
ymax=70,
xtick={0,7,14,21,28},
ytick={-10,10,30,50,70},
scaled x ticks=false,
xlabel={Time (sec)},
xlabel style={yshift=0.8mm},
ylabel={},
grid=major,
major grid style={gray!25},
tick label style={font=\scriptsize},
label style={font=\footnotesize},
legend style={
at={(0.97,0.97)},
anchor=north east,
font=\fontsize{6.5}{6.5}\selectfont,
draw=black,
fill=white,
inner sep=1.2pt,
row sep=0pt,
legend image post style={xscale=0.55}
},
xticklabel style={/pgf/number format/1000 sep={}},
axis line style={black},
tick style={black}
]
\addplot[color=cyan!60!blue,line width=1.0pt,mark=none]
table[x=time,y=u_z]{fig3_left.dat};
\addplot[color=orange,line width=1.0pt,mark=none]
table[x=time,y=u_x]{fig3_left.dat};

\legend{
$u\!=\!-\!K_z^{*}z$,
$u\!=\!-\!K^{*}x$}
\end{axis}
\end{tikzpicture}

\vspace{-2.6mm}
\makebox[1.12\linewidth][c]{\footnotesize (a) Input signal $u(t)~({\rm deg})$}
\phantomcaption
%\caption{}
\label{fig:u_compare}
\end{subfigure}
\hspace{-3mm}%
\begin{subfigure}{0.498\linewidth}
\centering
\begin{tikzpicture}
\begin{axis}[
width=1.16\linewidth,
height=4.65cm,
xmin=0,
xmax=28,
ymin=-10,
ymax=70,
xtick={0,7,14,21,28},
ytick={-10,10,30,50,70},
scaled x ticks=false,
xlabel={Time (sec)},
xlabel style={yshift=0.8mm},
ylabel={},
grid=major,
major grid style={gray!25},
tick label style={font=\scriptsize},
label style={font=\footnotesize},
legend style={
at={(0.97,0.97)},
anchor=north east,
font=\fontsize{8}{10}\selectfont,
draw=black,
fill=white,
inner sep=1.2pt,
row sep=0pt,
legend image post style={xscale=0.55}
},
xticklabel style={/pgf/number format/1000 sep={}},
axis line style={black},
tick style={black}
]
\addplot[color=cyan!60!blue,line width=1.0pt,mark=none]
table[x=time,y=y_z]{fig3_right.dat};
\addplot[color=orange,line width=1.0pt,mark=none]
table[x=time,y=y_x]{fig3_right.dat};

% \legend{
% $y(t)$ under $u\!=\!-\!K_z^{*}z$,
% $y(t)$ under $u\!=\!-\!K^{*}x$
% }
\end{axis}
\end{tikzpicture}

\vspace{-2.6mm}
\makebox[1.05\linewidth][c]{\footnotesize (b) Output signal $y(t)~({\rm deg/s})$}
\phantomcaption
%\caption{}
\label{fig:y_compare}
\end{subfigure}
\vspace{-4mm}
\caption{The input and output trajectories of \eqref{dyn:LTI}.}
\label{Fig4}
\end{figure}

%%%%
%%%%-----------------
\section{Conclusion}

This paper presented a data-driven optimal output-feedback control approach for continuous-time LTI systems with unknown matrices. Using a canonical non-minimal realization constructed through Kreisselmeier's adaptive filter, the  linear quadratic control problem was reformulated in an augmented coordinate driven by  input-output data. We established an explicit relation between the optimal gain of the augmented problem and the optimal state-feedback gain of the original plant, which enabled the design of a data-driven VI algorithm.  
The effectiveness of the proposed method was demonstrated via simulations on a linearized F-16 aircraft model.

\bibliographystyle{IEEEtran}
\bibliography{references.bib}

\end{document}